\documentclass{article}

\input{header.aux}

\newcommand{\ang}[1]{\left< #1 \right>}
\newcommand{\vols}{{\mathrm{VOL}^*}}

\title{Fractional Poincar\'{e} and logarithmic Sobolev inequalites for measure spaces}
\author{Philip T. Gressman}
\begin{document}
\maketitle
\begin{abstract}
We prove generalizations of the Poincar\'{e} and logarithmic Sobolev inequalities corresponding to the case of fractional derivatives in measure spaces with only a minimal amount of geometric structure.  The class of such spaces includes (but is not limited to) spaces of homogeneous type with doubling measures.  
Several examples and applications are given, including Poincar\'{e} inequalities for graph Laplacians, Fractional Poincar\'{e} inequalities of Mouhot, Russ, and Sire \cite{mrs2011}, and implications for recent work of the author and R. M. Strain on the Boltzmann collision operator \cite{gs2010,gs2011,gs2010III}.
\end{abstract}
\section{Introduction}
\subsection{Background}
The purpose of this note is to prove an analogue of the Poincar\'{e} inequality in an abstract setting.  Specifically, the goal is to establish an inequality of the form
\begin{equation} \int_X |f(x)|^p d \mu(x) \leq C \int_{X \times X} |f(x) - f(y)|^p K(x,y) d \mu(x) d \mu(y) \label{mrs0}
\end{equation}
for some suitable measure space $(X,\mu)$ and kernel $K(x,y)$ when $f$ belongs to an appropriate subspace of $L^p(X)$ (typically being the space of functions with vanishing mean with respect to $\mu$).  Inequalities of this type appear throughout the probability literature in the study of pure jump L\'{e}vy processes, where it is known (for example) that
\begin{equation} \int_{\R^d} |f(x)|^p d \mu(x) \leq \int_{\R^d \times \R^d} |f(x+h) - f(x)|^p d \nu(h) d \mu(x) \label{levy} \end{equation}
for any $p \geq 1$ when $\nu$ is the L\'{e}vy measure, $\mu$ is the distribution of the corresponding L\'{e}vy process at time $t=1$, and $\int f d \mu = 0$.  (In that context, \eqref{levy} is a special case of so-called $\Phi$-entropy estimates; see, for example, Wu \cite{wu2000} corollary 4.2 or  Chafa\"{\i} \cite{chafai2004} theorem 5.2 and references in  Gentil and Imbert \cite{gi2008} or Mouhot, Russ, and Sire \cite{mrs2011}).  Such inequalities are of intrinsic interest outside the probability literature as well by virtue of the analogy between the right-hand side of \eqref{levy} and the Gagliardo or Taibleson seminorms for fractional regularity spaces including the fractional $L^p$-Sobolev spaces and related constructions \cite{taibleson1963,stein1962,ams1963,gagliardo1958,af2003}.  Anisotropic versions of such constructions also appear via the space $N^{s,\gamma}$ recently introduced by the author and R. M. Strain to the study of the Boltzmann equation \cite{gs2010,gs2011,gs2010III}.

The direct study of general inequalities of the form \eqref{mrs0} was recently initiated by Mouhot, Russ, and Sire \cite{mrs2011}, who established the inequality
\[ \int_{\R^d} \!  |f(x)|^2 (1 + |\nabla V(x)|^\alpha) e^{-V(x)} dx \leq C \!  \int_{\R^d \times \R^d} \! \! \! \frac{|f(x) - f(y)|^2}{|x-y|^{d+\alpha}} e^{-\delta|x-y| -V(x)} dx dy \]
for all $f$ with $\int f e^{-V} = 0$ and some fixed $C,\delta$ when $e^{-V} \in L^1(\R^d)$ and $V \in C^2(\R^d)$ satisfies 
\begin{equation} s |\nabla V(x)|^2 - \Delta V(x) \rightarrow \infty \mbox{ as } |x| \rightarrow \infty \label{cond} \end{equation}
for some $s \in [0,\frac{1}{2})$.  There are two features of this inequality which are of particular interest when contrasted with the earlier result \eqref{levy} in the probabilistic case.  The first is that the result of Mouhot, Russ, and Sire is much more general in the sense that the measures appearing on the left- and right-hand sides enjoy a certain level of independence (whereas either one of the measures $\mu$ and $\nu$ in \eqref{levy} uniquely specifies the other).  The second feature of interest is the appearance of the $|\nabla V(x)|^\alpha$ on the left-hand side of their fractional Poincar\'{e} inequality; this is the result of a natural self-improvement that occurs for the classical Poincar\'{e} inequality with a full derivative
(which holds for the measure $e^{-V(x)} dx$ by virtue of \eqref{cond}, as shown in \cite{bbcg2008} and elsewhere).

\subsection{The main theorems}

From here forward, $X$ will refer to a topological space equipped with a nonnegative, $\sigma$-finite Borel measure $\mu$.  It will also be assumed that $X$ and $\mu$ come with a family of balls of unit radius:  specifically, $U$ will be a distinguished open neighborhood of the diagonal in $X \times X$ giving rise to two families of balls:
\begin{align*}
B_x & := \set{y \in X}{ (x,y) \in U}, \\
B^*_y & := \set{x \in X}{(x,y) \in U}.
\end{align*}
When convenient, the notation $B^1_x$ will also be used to refer to the ball $B_x$, and the point $x$ will be called the center point.  For any integer $n \geq 2$, let $B^n_x$ be the union of all balls whose center point is contained in $B^{n-1}_x$, i.e., $y \in B^n_x$ if there exist $y_1,\ldots,y_{n-1}$ such that all pairs $(x,y_1), (y_1,y_2), \ldots, (y_{n-1},y)$ are contained in $U$.  We assume that there are finite, positive constants $C$ and $\lambda_0$ such that
\begin{equation}
\mu(B_x^n) \leq C \lambda_0^n \mu(B_x) \label{subexp}
\end{equation}
holds for all $x \in X$ and all positive integers $n$, and the constant $\lambda_0$ will be called the growth constant of $X$ (we will implicitly assume that $\mu(B_x)$ is neither zero nor infinite for any $x$).  Note that the growth constant is finite when the balls $B$ arise as unit balls in a space of homogeneous type with a doubling measure.  In the specific case of Euclidean balls of radius $1$ and Lebesgue measure, the growth constant $\lambda_0$ may be taken arbitrarily close to $1$ (but always strictly greater).  Doubling measures, however, are not the only situation when the constant is finite; one may also take the balls $B$ to be unit balls with respect to a metric when the volume growth of metric balls is at most exponential.

With the structure given above, let $W$ and $W_+$ be nonnegative, measurable weight functions on $X$ such that $W_+(x) \geq W(x)$.  This pair $(W,W_+)$ will be called admissible when there exist positive constants $\lambda,\epsilon,s$ with $\lambda > \lambda_0$ and $s \in [0,1)$ so that
\begin{equation} 
\left[ 
W_+(x) \right]^s \epsilon \mu(B_y) \leq \int_{B_x^*} \left[ W(z) \right]^s \chi_{W(z) \geq \lambda W_+(x)} d \mu(z)
\label{connect}
\end{equation}
(the notation $\chi_E$ will be used for the indicator function of $E$) for any pair $(x,y) \in U$, provided that $x$ does not belong to some exceptional set $X_0$. When this exceptional set $X_0$ is nonempty, it will be assumed that it has diameter at most $1$ (meaning that $X_0 \times X_0 \subset U)$ and that $W_+(x) \mu(B_y) \leq C W(y) \mu(X_0)$ for all $x,y \in X_0$ and some fixed constant $C$.   Note that in the case $X_0 = \emptyset$ these conditions are trivially satisfied.

In this note, the question of sharp constants will not be addressed; to simplify matters, we will use the convention that $A \lesssim B$ means that there is an implied constant $C$ such that $A \leq C B$ holds uniformly over some specified family of parameters (for example, when $A$ and $B$ depend on an arbitrary function $f$, the statement $A \lesssim B$ will mean that the same constant $C$ may be used for any $f$).

With these definitions in place, the first main result is as follows:
\begin{theorem}
Let $(X,\mu,U)$ be as above and let $(W,W_+)$ be admissible.  Then for any measurable function $f$ such that $\int_{X_0} f d \mu = 0$ and $f \rightarrow 0$ uniformly as $W_+ \rightarrow \infty$ (meaning that $|f(x)| \leq \epsilon$ if $W_+(x)$ is sufficiently large), then for any fixed $p \in [1,\infty)$ one has \label{poincarethm}
\begin{equation} \int |f|^p W_+ d \mu \lesssim \int_X   \left[ \frac{1}{\mu(B_x)}  \int_{B_x} |f(x) - f(y)|^p d \mu(y)  \right] W(x) d \mu(x) \label{mainest}
\end{equation}
for some implied constant that is independent of $f$.
\end{theorem}

The second main result demonstrates that two-weight Poincar\'{e} inequalities like \eqref{mainest} automatically imply corresponding generalized logarithmic Sobolev inequalities.  These logarithmic Sobolev inequalities correspond to fractional versions of the sort of inequalities proved recently by Lott and Villani \cite{lv2009} and Sturm \cite{sturm2006I,sturm2006II}.
In this case it will be further assumed that the space $X$ is equipped with a nested sequence open sets $X \times X \supset U_0 \supset U_1 \supset U_2 \supset \cdots$ such that $U_0$ that the intersection of $U_j$ over all $j$ is the diagonal in $X \times X$.  The index $j$ is thought of as a dyadic parameter (with decreasing radius as $j$ increases), so that $U_0$ will be thought of as corresponding to ``unit'' balls.  The precise nature of these balls, however, is not important beyond the definitions just given.  To clearly distinguish these balls from the balls $B^j_x$, we will use the notation
\[ U_j(x) := \set{y \in X}{ (x,y) \in U_j} \mbox{ and } U_j^*(y) := \set{x \in X}{(x,y) \in U_j}. \]
For any pair of points $(x,y)$, let $\mathrm{VOL}^*(x,y)$ be defined to be the infimum of $\mu(U_{j+1}^*(y))$ over all $j$ such that $x \in U_j^*(y)$ (and note that $\mathrm{VOL}^*(x,y)$ will be undefined if $y$ does not belong to the unit ball $U_0(x)$).
Finally, let
\begin{align}
K_{p,\psi} (x,y) & := \frac{\psi \left( \left[ \vols(x,y) \right]^{-\frac{1}{p}} \right)}{\vols(x,y)} + \frac{1}{\mu(U_0(x))},  \nonumber \\
 ||f||_{p,\psi}^p & := \int_{X} \left( \int_{U_0(x)} |f(x) - f(y)|^p   K_{p,\psi}(x,y) d \mu(y) \right) W(x) d \mu(x), \label{singnorm}
\end{align}
where $p \in [1,\infty)$ and $\psi$ is a nonnegative, nondecreasing function on $[0,\infty)$ such that $\psi(0) > 0$ and $\lim_{t \rightarrow \infty} \psi(t) = \infty$.
\begin{theorem}
For fixed $\psi$ satisfying the hypotheses above, suppose that there exists a nonnegative, nondecreasing function $\tilde \psi$ on $[0,\infty)$ such that
\begin{equation} \psi(xy) \lesssim \psi(x) + \tilde \psi(y) \label{slowgrow} \end{equation}
for all nonnegative $x,y$ and some fixed implicit constant. \label{logsobthm}
Suppose that the two-weight Poincar\'{e} inequality
\begin{equation} \int |f|^p W_+ d \mu \lesssim \int_X   \left[ \frac{1}{\mu(U_0(x))}  \int_{U_0(x)} |f(x) - f(y)|^p d \mu(y)  \right] W(x) d \mu(x) \label{themajortheorem}
\end{equation}
holds uniformly for all $f$ in some class of functions $\mathcal{F}$ (note specifically that the pair $(W,W_+)$ is not necessarily assumed to be admissible).
Suppose that $\psi$ is slowly growing and that the nonnegative weights $W,W_+$ satisfy the inequalities
\begin{align} W_+(x) & \gtrsim W(x) \left[ \psi \left( [\mu(U_0^*(x))]^{-\frac{1}{p}} \right) + \tilde \psi\left( [W(x)]^{-\frac{1}{p}} \right)\right]  \label{sobcond1} \\
W(x)  & \lesssim \frac{1}{\mu(U_j^*(x))} \int_{U_j^*(x)} W(y) d \mu(y) \label{subsol0}
\end{align}
uniformly for all $x \in X$.
Then there is a constant $c$  such that
\begin{equation}
\int_X \psi \left( \frac{ c |f(y)| }{ ||f||_{p,\psi}} \right) \left| \frac{ c |f(y)| }{ ||f||_{p,\psi}}  \right|^p W(y) d \mu(y) \leq 1
\label{logsob} \end{equation}
uniformly for all $f in \mathcal{F}$.
\end{theorem}

\subsection{Comments and examples}
In this section we record several comments and examples with the hope that the reader will find them illuminating.  \label{commentsec}
 
 1.  Note that there is no required regularity of the weights $W$ and $W_+$.  Instead, the admissibility condition \eqref{connect} is only sensitive to what one might think of as the bulk decay properties of the weights; heuristically, when $s=0$ and $W=W_+$, for example, \eqref{connect} will be true when there is a nontrivial fraction of the ball $B_x^*$ on which $W$ is at least as large as $\lambda W(x)$.  A somewhat cleaner but less general substitute for \eqref{connect} is given by
\begin{equation}  \left[ \lambda W_+(x) \right]^s \left[ \mu(B_x^*) + \epsilon \mu(B_y) \right] \leq \int_{B_x^*} \left[ W(z) \right]^s d \mu(z), \label{connectalt} \end{equation}
when $s \in (0,1)$ (as long as $s > 0$, the inequality \eqref{connectalt} trivially implies \eqref{connect}),
which further suggests that the main thrust of \eqref{connect} is that the average of the weight on balls $B_x^*$ should be uniformly proportionally larger than the value of the weight at the center of the ball.  In the specific case when $W := e^{-V}$, the inequality \eqref{connectalt} is, in some sense, a global analogue of the differential condition 
\begin{equation} s |\nabla V(x)|^2 - \Delta V(x) \geq \rho > 0, \label{differential}
\end{equation}
for some $s \in (0,1)$, which will be shown explicitly at the end of this note via lemma \ref{diflem}.  The point is that the conditions \eqref{connect} and \eqref{connectalt} are significantly less-restrictive hypotheses than \eqref{differential} since they concern only balls of a fixed scale, whereas \eqref{differential} implies related inequalities for balls at all scales.  A corollary of this implication is the following (also to be proved in the final section):
\begin{corollary}
Suppose $V \in C^2(\R^d)$.  Let $W(x) := e^{-V(x)}$ and \label{bigcor}
\[ W_+(x) := W(x) + \frac{1}{m(B_x^1)} \int_{B_x^1} W(y) dy \]
where $m$ and $dy$ refer to Lebesgue measure and $B_x^1$ is the Euclidean ball of radius $1$ centered at $x$.  Suppose also that $\int_{\R^d} e^{-V(x)} dy = 1$. 
If there is an $s \in (0,1)$ and $\rho > 0$ such that \eqref{differential}
holds when $|x| > R$, then
\[ \int_{\R^d} |f(x)|^p W_+(x) dx \lesssim \int_{\R^d} \left[ \int_{B_x^1} |f(x) - f(y)|^p dy \right] W(x) dx \]
holds uniformly for all $f$ with $\int f W = 0$ and any fixed $p \in [1,\infty)$.
\end{corollary}
This result should be compared to the result of Mouhot, Russ, and Sire \cite{mrs2011}.  It is perhaps also worth noting that it may be shown that \eqref{differential} holding outside a compact set is a sufficient condition for the classical Poincar\'{e} inequality as well; see \cite{bbcg2008} and \cite{villani2009}.  The dependence on only one scale allows for corollaries of \eqref{mainest} which are, in fact, entirely discrete.  If $X$ is a connected graph (finite or infinite) whose vertices have valence at most $k$ and $W$ decays exponentially rapidly in the distance to some fixed point (with the rate of decay depending on $k$), then \eqref{mainest} reduces when $p=2$ to the Poincar\'{e} inequality for the graph Laplacian with the neighborhood $U$ consisting of all pairs of vertices $x$ and $y$ which are connected by an edge.

2. The fact that the ball $B_x$ in \eqref{mainest} is generally compactly supported (e.g., in $\R^d$ with the usual balls) and that the measure on the right-hand side of \eqref{mainest} is absolutely continuous with respect to $\mu \times \mu$  means that the right-hand side of \eqref{mainest} will, in general, be substantially smaller than corresponding seminorms where the measure is singular and  decays slowly away from the diagonal.  In fact, if the inequality
\[ \int_{B_y^*} \frac{W(x)}{\mu(B_x)} d \mu(x) \lesssim W_+(y) \]
holds uniformly in $y$, then both sides of \eqref{mainest} can be quickly seen to be comparable.  Rapid decay away from the diagonal is especially crucial:  when $W(x) dx$ is a probability measure on $\R^d$, for example, Jensen's inequality dictates that
\[ \int \left| f(x) - \int f(y) W(y) d y \right|^p W(x) dx \leq \int \! \! |f(x) - f(y)|^p W(x) W(y) d x d y. \]
Thus if $W(x) \leq C e^{-2 \delta |x|}$ (which holds, for example, when $W$ is log concave, as shown in \cite{bbcg2008}) then $W(x) W(y) \leq C e^{-\delta |x-y|} \left[ W(x) + W(y) \right]$; consequently
\[\int_{\R^d} \left| f(x) \right|^p W(x) dx \leq 2 C \int_{\R^d \times \R^d} |f(x) - f(y)|^p W(x) e^{-\delta |x-y|} dx dy \]
holds immediately for all functions $f$ with mean zero.  In general, the decay away from the diagonal should be faster than the decay of the weight $W$ itself.

3. In practice, the exceptional set $X_0$ need not always have diameter $1$ for the theorem to hold; it is, however, necessary in such cases that one should be able to modify $W$ and $W_+$ up to bounded multiplicative factors so that the new pair of weights has a new exceptional set of small diameter.  This idea is illustrated in the following corollary:
\begin{corollary}
Suppose $\Omega \subset \R^d$ is open, bounded, and the $\frac{1}{2}$-neighborhood of $\Omega$ (all points with Euclidean distance strictly less than $\frac{1}{2}$ from $\Omega$) is connected.  Then \label{connectcor}
\[ \int_\Omega |f(x)|^p dx \lesssim \int_{ \Omega \times \Omega} |f(x) - f(y)|^p \chi_{|x-y| < 1} dx dy \]
holds uniformly for all $f$ with $\int_\Omega f = 0$ and any fixed $p \in [1,\infty)$.
\end{corollary}
\begin{proof}
Let $m(E)$ denote the Lebesgue measure of the set $E$.
Fix some point $x \in \Omega$ and let $O_1$ be any Euclidean ball of radius at most $\frac{1}{2}$ contained entirely within $\Omega$.  For $n \geq 2$, let
\[ O_n := \set{ y \in \R^d}{m(B_y \cap O_{n-1} \cap \Omega) > \frac{c}{n}}. \]
Clearly each set $O_n$ is open, and if $c$ is chosen small enough that $O_1 \subset O_2$ then by induction the sets $O_n$ will be nested and increasing.  Now suppose  $x \not \in {\mathcal O} := \bigcup_{n=1}^\infty O_n$.  By monotone convergence,
\[m(B_x \cap O_n \cap \Omega) \rightarrow m(B_x \cap {\mathcal O} \cap \Omega) \]
as $n \rightarrow \infty$.  Since $x \not \in {\mathcal O}$, necessarily $m(B_x \cap {\mathcal O} \cap \Omega) = 0$.  In particular, this means that the open set $B_x \cap {\mathcal O} \cap \Omega$ must be empty.  
Now suppose that there is a point $z \in \Omega \setminus {\mathcal O}$.  
Since $z \not \in {\mathcal O}$, $m(B_z \cap O \cap \Omega) = 0$, meaning that $z$ must be at least a distance $1$ away from ${\mathcal O} \cap \Omega$.  In turn, this means that the $\frac{1}{2}$-neighborhoods of $\Omega \cap {\mathcal O}$ and $\Omega \setminus {\mathcal O}$ must be disjoint.  Assuming that the $\frac{1}{2}$-neighborhood of $\Omega$ is connected then implies precisely that $\Omega \subset {\mathcal O}$.
Now since the closure of $\Omega$ is compact, there must be a finite $n$ for which $\Omega \subset O_n$.  For each $x \in \Omega$, let $V(x)$ equal the smallest $n$ for which $x \in O_n$.  The pair $(e^{-V}, e^{-V})$ is admissible with $\lambda = e$ and $X_0 = O_1$.  However $e^{-V}$ is bounded above and below on $\Omega$ 
so the corollary holds by virtue of the main theorem applied to the pair $(e^{-V}, e^{-V})$.
\end{proof}
This corollary should be compared to the classical case, in which the domain $\Omega$ must necessarily be connected and additional conditions like an interior cone condition or belong to the class of John domains (see \cite{lsy2003,hajlasz2001}).  Other counterexamples can be found by considering graph Laplacians on disconnected graphs (since, in this case, the null space of the Laplacian contains nonconstant functions).

4. Because of the nonlocality of the right-hand side of \eqref{mainest}, the magnitude of possible self-improvement of this inequality (i.e., the size of the ratio $W_+/W$) can be exponentially larger than what is possible in the classical Poincar\'{e} case or in the work of Mouhot, Russ, and Sire.  For example, applying the main theorem on $\R^d$ with $\mu$ equal to the Lebesgue measure and $B_x$ equal to the Euclidean ball of radius $1$, it is possible to achieve the inequality
\begin{equation} \int_{\R^d} |f(x)|^p e^{\epsilon s |x|^{s-1}} e^{-|x|^s} d x \lesssim \int_{|x-y| \leq 1}  |f(x) - f(y)|^p e^{-|x|^s} dx dy \label{bigex}
\end{equation}
for any fixed $\epsilon \in [0,1)$ and  $s \geq 1$ when $f$ has mean zero.  In most cases, the $W_+$ appearing in corollary \ref{bigcor} will, in the same way, be substantially larger than $(1+ |\nabla V(x)|^\alpha)e^{-V(x)}$ for any fixed power of $\alpha$.

5.  The inequality \eqref{mainest} can be used to establish Dirichlet-type fractional Poincar\'{e} inequalities as well as the usual Neumann variety.  As in the example above, we also have that
\[ \int_{\R^d} |f(x)|^p e^{\epsilon s |x|^{s-1}} e^{|x|^s} d x \lesssim \int_{|x-y| \leq 1}  |f(x) - f(y)|^p e^{|x|^s} dx dy \]
for any fixed $\epsilon \in [0,1), s \geq 1$ when $f \rightarrow 0$ as $x \rightarrow \infty$.  In the Neumann case (when $\int W_+ d\mu = 1$) the inequality \eqref{mainest} also holds under the usual assumption that $\int f W_+ d \mu = 0$ instead of $\int_{X_0} f d \mu = 0$ since,  by Minkowski's inequality,
\[ \left(\int  \left| f (x) - \int f W_+ d \mu \right|^p W_+(x) d \mu(x) \right)^\frac{1}{p} \leq 2 \left( \int \left| f(x) - c \right|^p W_+ d \mu \right)^{\frac{1}{p}} \]
for any function $f$ and any constant $c$. 

6. Another example of interest concerns the seminorms connected to the Boltzmann collision operator \cite{gs2010,gs2011,gs2010III}.  If one fixes $(d(v,v'))^2 = |v-v'|^2 + \frac{1}{4} (|v|^2 - |v'|^2)^2$ for $v,v' \in \R^d$ and defines $\ang{v} := \sqrt{1 + |v|^2}$, then
\[ \int_{\R^d} \! |f(v)|^2 \ang{v}^{\alpha} e^{- |v|^2} dv \lesssim \int_{d(v,v') \leq 1} \! \! |f(v) - f(v')|^2 \ang{v}^{\alpha+1} e^{-|v|^2} dv dv' \]
when $f$ has mean zero and $\alpha$ is any fixed real number, giving a Poincar\'{e} inequality between the norm $L^2_{\gamma+2s}$ and the seminorm term defining $N^{s,\gamma}$.

7.  Regarding theorem \ref{logsobthm}, the examples of $\psi$ and $\tilde \psi$ of particular interest include
\[ \psi(x) := \log^\alpha ( e + x) \mbox{ and } \tilde \psi(x) := \psi(x), \]
\[ \psi(x) := e^{c \log^{\alpha} (e + x)} \mbox{ and } \tilde \psi(x) := e^{c \log^{\frac{\alpha}{1-\alpha}}(e+x)}, \]
where $\alpha \geq 0$ in the first pair and $c \geq 0$, $\alpha \in [0,1)$ in the second.  As in comment 4, tremendous gains in the weight $W_{+}$ translate to tremendous gains in the generalized logarithmic Sobolev inequality.  Returning to \eqref{bigex}, in dimension $d$ we have that
\[ ||f||_{L^p e^{c \log^{\alpha} L}(e^{-|x|^s} dx)} \lesssim \left( \int_{|x-y| \leq 1} |f(x) - f(y)|^p ~ \frac{e^{c \log^{\alpha} (1 + |x-y|^{-1})}}{|x-y|^d} dx dy \right)^{\frac{1}{p}} \]
provided $\alpha \leq \frac{s-1}{2s-1}$ and $c$ is sufficiently small (depending on $s, \alpha,d$, and $p$).  The particular case when $p = s = 2$ should be contrasted with the classical logarithmic Sobolev inequality of Gross \cite{gross1975}, who showed embedding into $L^2 \log L$ (with sharp constant equal to $1$ independent of dimension, a feat with regard to which the present note cannot compete) and nonembedding into $L^2 \log L \log \log L$.  Recent logarithmic Sobolev inequalities due to Lott and Villani \cite{lv2009} and Sturm \cite{sturm2006II} are also implied in many cases by theorem \ref{logsobthm}.  Roughly speaking, the common approach of these earlier works is that they assume more structure and deal with smoother functions (corresponding in the appropriate sense to a ``full'' derivative) and are, in turn, able to make meaningful statements about the various constants involved.  The present paper, by contrast, assumes minimal geometric structure and considers functions with lower degrees of smoothness at the cost of poor understanding of the constants.

8.  While generally the singularity of $K_{p,\psi}$ is only slightly stronger than the natural scaling $( \vols(x,y))^{-1}$ (since $\psi(t)$ will be $o(t^\epsilon)$ for any $\epsilon > 0$), it is nevertheless possible to do slightly better if one makes additional assumptions about the measure $\mu$ (specifically corresponding to Ahlfors-David regularity).  For example, if $X$ is a compact subset of $\R^d$ which possesses a finite Borel measure $\mu$ for which $r^{-\alpha} \mu(B_x(r))$ is bounded uniformly above and below by nonzero constants for all $x$ and $r$ (here $B_x(r)$ is the usual Euclidean ball of radius $r$ centered at $x$), then
\[ ||f ||_{L^p \log^s L(X)}^p \lesssim \int_{X \times X} |f(x) - f(y)|^p \frac{\log^{s-1}(1 + |x-y|^{-1})}{|x-y|^\alpha} d \mu(x) d \mu(y) \]
provided $\int_{X} f d \mu = 0$ (with $p \in [1,\infty)$ and $s \geq 0$ both fixed).  This result may be obtained using a minor refinement of lemma \ref{logsoblemma}; a sketch of the details appears at the end of section \ref{logsobsec}.


\section{The proof theorem \ref{poincarethm}}

The main tool used in the proof of \eqref{mainest} is a generalization of the method of Lyapunov functions (not unlike the approach used by Bakry, Barthe, Cattiaux, and Guillin \cite{bbcg2008}; Bakry, Cattiaux, and Guillin \cite{bcg2008}; or by Cattiaux, Guillin, Wang, and Wu \cite{cgww2009}).  It is well-known that for any linear mapping $T$ between Hilbert spaces, the largest eigenvalue of $T^* T$ equals the square of the norm of $T$; what may be perhaps less well-known is that when $T$ maps nonnegative functions to nonnegative functions, one may make a similar assertion about $T$ as a mapping on $L^p$ spaces with the role of $T^* T$ played by the nonlinear operator $f \mapsto T^* ( T f)^{p-1}$.  The proof is elementary:
\begin{lemma}
Let $X$ and $Y$ be topological spaces equipped with nonnegative measures $\mu$ and $\nu$, respectively. \label{fancylemma}
Suppose $T$ is a linear operator which maps some vector space of measurable functions on $Y$ to measurable functions on $X$ (in both cases, modulo equivalence almost everywhere).  Assume also that $T$ satisfies the following properties:
\begin{enumerate}
\item The operator $T$ maps nonnegative functions to nonnegative functions.
\item If $f_1,f_2,\ldots$ is a pointwise nondecreasing sequence of nonnegative functions in the domain of $T$ and the function $\sup_n f_n$ is also in the domain of $T$ then $T(\sup_n f_n) \leq \sup_n T f_n$ almost everywhere.
\item There is a nonnegative function $f$, strictly positive and finite $\nu$-a.e., such that $f \chi_{E}$ is in the domain of $T$ for any measurable set $E$ and that
\begin{equation}
\left( \int | T f|^{p-1} | T (f \chi_E)| d \mu \right)^{\frac{1}{p}} \leq C \left( \int f^p \chi_E d \nu \right)^{\frac{1}{p}} .\label{mega}
\end{equation}
holds for some finite constant $C$ (where both sides are allowed in some cases to equal $+\infty$) and some fixed exponent $1 \leq p < \infty$.  The function $f$ will be called a Lyapunov function.
\end{enumerate}
Then $T$ extends uniquely to a linear mapping from $L^p(Y)$ to $L^p(X)$ which satisfies $|| T g ||_p \leq C ||g||_p$ for the  same constant $C$ appearing in \eqref{mega}.
\end{lemma}
\begin{proof}
Note that the case $p=1$ is essentially trivial.  For $p > 1$, fix any positive integer $n$ and let $E_1,\ldots,E_n$ be disjoint $\nu$-measurable sets in $Y$  such that $\int_{E_i} |f|^p d \nu < \infty$ for each $i$.  Now define $T_n : \C^n \rightarrow L^p(X)$ by
\[ T_n (u) := \sum_{i=1}^n u_i T (f \chi_{E_i}) \]
for $u := (u_1,\ldots,u_n)$.  By continuity in $u$ and compactness, $|| T_n u ||_{L^p(X)}^p$ achieves a maximum on the unit ball $|| f \sum_{i=1}^n u_i \chi_{E_i}||_{L^p(Y)}^p = 1$.  Moreover, since $T (f \chi_E)$ is nonnegative (almost everywhere), it suffices to assume that each entry of $u_i$ is a nonnegative real number since $|T_n u| \leq T_n |u|$.  Now at this extreme point $u$, the Euclidean gradient (with respect to $u$) of the quantity $|| T_n u ||_{L^p(X)}^p$ must be proportional to the Euclidean $u$-gradient of $|| f \sum_{i=1}^n u_i \chi_{E_i}||_{L^p(Y)}^p$ 
(for if not, there would be a direction tangent to the unit sphere $|| f \sum_{i=1}^n u_i \chi_{E_i}||_{L^p(Y)}^p = 1$ in which $|| T_n u ||_{L^p(X)}^p$ increases).  Differentiating with respect to the coordinate $u_i$ gives that
\begin{equation} \int \left| \sum_{j=1}^n u_j T ( f \chi_{E_i}) \right|^{p-1} T (f \chi_{E_i}) d \mu = \lambda u_i^{p-1} \int_{E_i} |f|^p d \nu \label{mid}
\end{equation}
for each $i$ and some fixed constant $\lambda$.  Note in particular that the interchange of differentiation and integration is permitted by the Lebesgue dominated convergence theorem since \eqref{mega} and the nonnegativity of $T$ imply that $T \chi_{E_i} \in L^p(X)$ for each $i$ because each $E_i$ was chosen specifically to make the right-hand side of \eqref{mega} finite (and likewise for the differentiation taking place on the right-hand side of \eqref{mid}).  In fact, multiplying \eqref{mid} by $u_i$ and summing over $i$ gives that $\lambda$ must exactly equal the maximum value of $||T_n u||^p$ which we will call $||T_n||^p$.  Now consider the mapping $S$ on nonnegative $n$-tuples given by
\[ (S v)_i := \left(\left( \int_{E_i} |f|^p d \nu \right)^{-1} \int \left| \sum_{j=1}^n v_j T ( f \chi_{E_i}) \right|^{p-1} T (f \chi_{E_i}) d \mu \right)^{\frac{1}{p-1}}. \]
Without loss of generality one may assume that $\int_{E_i} |f|^p d \nu > 0$ for each $i$ since \eqref{mega} and nonnegativity would imply in the case of equality with zero that $T f \chi_{E_i} = 0$ $\mu$-almost everywhere.
By construction $S u = ||T_n||^{p/(p-1)} u$ for the extreme point $u$.  This mapping $S$ is homogeneous of degree $1$, so clearly $S^N u = ||T_n||^{pN/(p-1)} u$ for any positive integer $N$.  Moreover by \eqref{mega} and nonnegativity, one also has that $S^N {\mathbf 1} \leq C^{pN/(p-1)} {\mathbf 1}$, where ${\mathbf 1}$ is the vector in $\R^n$ with all entries equal to $1$.  Thus for any $i$ one has that
\[ ||T_n||^{\frac{pN}{p-1}} u_i = (S^N u)_i \leq (\sup_{j} u_j) (S^N {\mathbf 1})_i \leq (\sup_{j} u_j) C^{\frac{pN}{p-1}} {\mathbf 1}. \]
As both sides are finite, choosing $i$ so that $u_i \neq 0$, taking $N$-th roots and letting $N \rightarrow \infty$ yields the inequality $||T_n|| \leq C$.
Consequently, for any nonnegative simple function $g$ with $\int |f|^p g d \nu < \infty$ it must be the case that $||T (gf)||_{L^p(X)} \leq C || g f ||_{L^p(Y)}$ for the same constant $C$ appearing in \eqref{mega}.  But functions of the form $g f$ with $g$ simple are dense in $L^p(Y)$ since $f$ is positive and finite almost everywhere.  In fact, a general nonnegative function $h$ may be approximated by taking $g_1,g_2,\ldots$ to be a sequence of simple functions converging monotonically to $h/f$ everywhere.  Then $f g$ converges to $h$ almost everywhere, and by the constraint $T \sup_{n} g_n f \leq \sup T (g_n f)$, $L^p$ boundedness follows (the supremum inequality in particular ensures that when the domain of $T$ is extended by completion, it agrees with the existing definition of $T$).
\end{proof}

Supposing that the pair $(W,W_+)$  is admissible, let $E \subset X \times X$ consist of all pairs $(x,y)$ for which $x \not \in X_0$, $y \in B_x^*$,  and $W(y) \geq \lambda W_+(x)$.  Now let 
\[ P(x,y) :=  \chi_E(x,y) \left[ W(y) \right]^s \left( \int_{B^*_x} \chi_E(x,z) \left[ W(z) \right]^s d \mu(z) \right)^{-1}.  \]
The relevant properties of $P$ to record at this point are 
\begin{align} \int P(x,y) d \mu(y) & = \chi_{X_0^c}(x), \label{stoceq}  \\
P(x,y) & \lesssim ( \mu(B_y) )^{-1} \chi_E(x,y) \left[ \frac{ W(y)}{ W_+(x)} \right]^s \label{linfty} \end{align}
for all $x,y \in X$ (the former is a trivial consequence of Fubini and the latter is precisely the content of \eqref{connect}).
Now consider the operators given by
\begin{align*}
 T f(x) & :=  \int P(x,y) f(y) d \mu (y), \\
 S_1 g(x) & := \int P(x,y) g(x,y) d \mu(y) + \frac{\chi_{X_0}(x)}{\mu(X_0)} \int_{X_0} g(x,y) d \mu(y), \\
S_n g & :=  S_1 g + T S_{n-1} g,
\end{align*}
for nonnegative measurable functions $f$ and $g$ on $X$ and $X \times X$, respectively, and all integers $n \geq 2$.  The action of  $S_n$ from $L^p(X \times X)$ to $L^p(X)$ will be studied by means of lemma \ref{fancylemma} with the Lyapunov function being given by
\[ F(x,y) := \begin{cases} \left[ W_+(x) \right]^{-\delta} - \left[ W_+(y) \right]^{-\delta} & x \not \in X_0 \\
 \left[ W_+(x) \right]^{-\delta} & x \in X_0 
\end{cases}
\]
for some small positive $\delta$ to be chosen momentarily.  Since $W_+(y) \geq W(y) \geq \lambda W_+(x)$ on the support of $E$, the function $F(x,y)$ is finite and positive everywhere, and one has the series of inequalities
\[ \left[W_+(x) \right]^{-\delta} \geq F(x,y) \geq (1 - \lambda^{-\delta}) \left[W_+(x) \right]^{-\delta} \]
for $(x,y) \in E$.  Moreover, the identity
\begin{align*} S_1 F(x) & = \chi_{X_0^c}(x) \left[ W_+(x) \right]^{-\delta} - T (W_+^{-\delta})(x) + \chi_{X_0}(x) \left[ W_+(x) \right]^{-\delta} \\ & = \left[ W_+(x) \right]^{-\delta} - T(W_+^{-\delta})(x) 
\end{align*}
(verified directly from the definition of $S_1$) gives by induction that $S_n F(x) = \left[ W_+(x) \right]^{-\delta}  - T^n (W_+^{-\delta})(x)$ for all $n$.  In particular, \begin{equation} S_n F(x) \leq \left[ W_+(x) \right]^{-\delta} \label{lyapunov1} \end{equation} for all $n$ and all $x \in X$.

Next, a closer look at $T$ is in order.  The inequalities \eqref{stoceq} and \eqref{linfty} establish the estimates
\begin{align*}
|T f(x)| & \leq \mathop{\mathrm{ess.sup}}_{y \in X} |f(y)| \chi_{E}(x,y), \\
|T f(x)| & \lesssim \int \chi_{E}(x,y) (\mu(B_y))^{-1} |f(y)|  \left[ \frac{ W(y)}{ W_+(x)} \right]^s d \mu(y), 
\end{align*}
both valid for every $x \in X$ (note that the second estimate will play the role of hypercontractivity).  Consequently, for any $n \geq 1$ we have that
\[ |T^n f(y_0)| \lesssim \sup_{y_1,\ldots,y_{n-1}} \int \left(\prod_{j=1}^n \chi_E(y_{j-1},y_j) \right) \frac{ |f(y_n)|}{\mu(B_{y_n})}  \left[ \frac{ W(y_n)}{ W_+(y_{n-1})} \right]^s d \mu(y_n). \]
Note that all pairs $(y_n,y_{n-1}),\ldots,(y_1,y_0)$ belong to $U$ by virtue of the definition of $E$, so the product of indicator functions is bounded above by $\chi_{B^n_{y_n}}(y_0)$.
By the admissibility of $(W,W_+)$, it is also true that $W_+(y_{n-1}) \geq \lambda^{n-1} W_+(y_0)$ on the support of the product of indicator functions. In particular, then, 
\[ \sum_{n=1}^\infty |T^n f(x)| \lesssim \int \sum_{n=1}^\infty \chi_{B_y^n}(x) \left[ \lambda^{-n} \frac{W_+(y)}{W_+(x)} \right]^{s'} \frac{|f(y)|}{\mu(B_y)} d \mu(y) \]
provided $s' \geq s$.  Now choose $\delta$ sufficiently small so that $s' := 1 - \delta(p-1) \geq s$.  Multiplying both sides by $\left[ W_+(x) \right]^{1-\delta(p-1)}$ and integrating with respect to $x$ gives that
\begin{align*}
 \int  & \sum_{n=1}^\infty |T^n f(x)|  \left[ W_+(x)\right]^{1- \delta(p-1)} d \mu(x) \\
& \lesssim  \int \sum_{n=1}^\infty \frac{\mu(B^n_y)}{\mu(B_y)} \lambda^{-n(1- \delta(p-1))} |f(y)|   \left[ W_+(y)\right]^{1- \delta(p-1)} d \mu(y) \\
& \lesssim \sum_{n=1}^\infty \lambda_0^n \lambda^{-n(1-\delta(p-1))} \int  |f(y)| \left[ W_+(y)\right]^{1- \delta(p-1)}  d \mu(y)
\end{align*}
by \eqref{subexp}, assuming that $\delta$ is also chosen small enough that $\lambda^{1-\delta(p-1)} > \lambda_0$.
Combining this with \eqref{lyapunov1} and the observation that $S_n g = S_1 g + T S_1 g + \cdots + T^{n-1} S_1 g$ gives that
\begin{align*}
\int |S_n & F|^{p-1}(x) S_n(F \chi_G)(x) W_+(x) d \mu(x) \\ & \lesssim \int  S_1 (F \chi_G)(y) \left[ W_+(y) \right]^{1-\delta(p-1)} d \mu(y) \\
& \lesssim \int  \frac{\chi_{U}(z,y)}{\mu(B_z)} \chi_G(y,z) \left[ F(y,z) \right]^p W(z) d \mu(y) d \mu(z) \\
& \hspace{20pt} +  \frac{1}{\mu(X_0)}  \int_{X_0 \times X_0}  \chi_G(y,z) \left[F(y,z) \right]^p W_+(y) d \mu(y) d \mu(z)
\end{align*}
(since, in particular, $\left[ W_+(y) \right] \left[ W(z) / W_+(y) \right]^s \lesssim W(z)$ for $(y,z) \in E$ when $s < 1$).
By the assumption that $W_+(y) \mu(B_z) \leq C W(z) \mu(X_0)$, it follows that
\begin{align*}
\int |S_n F|^{p-1}(x) & S_n(F \chi_G)(x) W_+(x)d \mu(x)  \\ & \lesssim \int    \frac{\chi_{U}(z,y)}{\mu(B_z)} F^{p}(y,z) \chi_G(y,z) W(z) d \mu(y) d \mu(z)
\end{align*}
holds uniformly for all $G$ and all $n$.  For each $n$ we may thus apply lemma \ref{fancylemma} since $S_n$ is {\it a priori} defined on all nonnegative measurable functions (and the monotonicity requirement of the lemma follows immediately from the Lebesgue monotone convergence theorem).  Thus we have that
\[ \int |S_n g(x)|^p W_+(x) d \mu(x) \lesssim \int_X \int_{B_z} \frac{ |g(y,z)|^p}{\mu(B_z)} d \mu(y) W(z) d\mu(z) \]
for all fixed $p \in [1,\infty)$, uniformly in $g$ and in $n$.

Now suppose that $f$ is any measurable function on $X$ which goes to zero uniformly as $W_{+} \rightarrow \infty$ and satisfies $\int_{X_0} f d \mu = 0$.  Let $\Delta_f(x,y) := |f(x)-f(y)|$ and consider the action of the $S_i$'s on this function $\Delta_f$.  In the case of $S_1$, we have
\begin{align*}
 S_1 \Delta_f(x) &  = \int P(x,y) |f(x) - f(y)| d \mu(y) \\
 & \hspace{20pt} + \frac{1}{\mu(X_0)} \chi_{X_0}(x) \int_{X_0} |f(x) - f(y)| d \mu(y) \\
& \geq \left| f(x) - T f(x) \right| \chi_{X_0^c}(x) + \left| f(x) - \frac{1}{\mu(X_0)} \int_{X_0} f(y) d \mu(y) \right| \chi_{X_0}(x) \\
& \geq |f(x) - T f(x)|
\end{align*}
since $T f(x) = 0 = \int_{X_0} f d \mu$ when $x \in X_0$.
It is trivial, then, to show by induction that
\[ S_n \Delta_f(x) \geq |f(x) - T^n f(x)| \]
since
\begin{align*}
S_n\Delta_f(x)  & \geq |f(x) - T f(x)| + \int P(x,y) |f(y) -  T^{n-1} f(y)| d \nu(y) \\
 & \geq |f(x) - {T} f(x)| + \left| T f(x) - {T}^{n} f(x) \right|  \geq |f(x) - {T}^n f(x)|. 
\end{align*}
We note, however, that by virtue of Fubini, 
\[ T^n f(x) = \int P^n(x,y) f(y) d\mu(y) \]
for some function $P^n$ with $\int P^n(x,y) d \mu(y) \leq 1$.  Moreover, for fixed $x$, the support of $P^n(x,\cdot)$ is necessarily contained in the union of $X_0$ and the set of points $y$ where $W_{+}(y) \geq \lambda^n W_+(x)$.  In particular,
\[ |T^n f(x)| \leq T^{n-1} (\chi_{W_+(\cdot) \geq \lambda^{n-1}W_+(x)} | T f|)(x) + T^{n-1}( \chi_{X_0} |T f|)(x). \]
However, since $T f(x) = 0$ on $X_0$, it simply follows that
\[ |T^n f(x)| \leq \sup_y |f(y)| \chi_{W_+(y) \geq \lambda^{n-1}W_+(x)} \]
which goes to zero as $n \rightarrow \infty$ for each $x$ by assumption on $f$ (outside of a set of $W_+ d \mu$ measure zero where $W_+(x) = 0$).  Consequently, by Fatou's lemma and the boundedness of the $S_n$'s, it follows that
\begin{align*}
 \int |f(x)|^p W_+(x) d \mu(x) & \leq \liminf_{n \rightarrow \infty} \int |S_n \Delta_f(x)|^{p} W_+(x) d \mu(x) \\ & \lesssim \int_X \int_{B_z} \frac{ |\Delta_f(y,z)|^p}{\mu(B_z)} d \mu(y) W(z) d\mu(z) 
 \end{align*}
 which is exactly the asserted conclusion of the main theorem.

\section{The proof of corollary \ref{bigcor}}
This final section begins with an explanation of the connection between the condition \eqref{connectalt} and \eqref{differential}.  This connection is provided by lemma \ref{diflem} below after one makes the observation that when $W=e^{-V}$,
\[ \frac{\Delta [W (x)]^s}{[W(x)]^s} = s \left( s |\nabla V(x)|^2 - \Delta V(x) \right). \]
\begin{lemma}
Suppose $F \in C^2(\R^d)$ is nonnegative and satisfies the inequality \label{diflem}
\begin{equation} \Delta F(x) \geq \rho F(x) \label{subsol} \end{equation}
for some $\rho > 0$ at every point of the ball $B_x^r := \set{y}{|x-y| < r}$.  Then
\begin{equation} \frac{1}{m(B_x^t)} \int_{B_x^t} F(y) dy \geq \left( 1 + \frac{\rho t^2}{2 (d+2)} \right) F(x) \label{conseq}
\end{equation}
for any $0 < t \leq r$, where $m$ and $dy$ refer to Lebesgue measure.  Consequently, if \eqref{subsol} holds at every point in $|x| > R$ and $F \in L^1(\R^d)$ then there is a universal constant $c$ for which $F(x) e^{c |x| \sqrt{\rho /d} } \in L^\infty (\R^d)$.
\end{lemma}
\begin{proof}
Without loss of generality, the point $x$ may be taken to reside at the origin in $\R^d$.
The proof follows the same lines as the usual proof for the mean value property of harmonic functions.  For any $t \in (0,r)$, let $\phi(t)$ be given by
\[ \phi(t) := \int_{{\mathbb S}^{d-1}} F(t \sigma) d \sigma \]
where $d \sigma$ is the uniform measure on the unit sphere in $\R^d$ with the normalization induced by Lebesgue measure.  On $(0,r)$, the function $\phi$ is clearly twice differentiable.  Differentiating with respect to $t$ and applying the divergence theorem and \eqref{subsol} gives
\begin{align*}
\frac{d \phi}{dt}(t)  & = \int_{{\mathbb S}^{d-1}} \sigma \cdot (\nabla F)(t \sigma) d \sigma = t \int_{B_0^1} (\Delta F)(t x) dx \geq  \rho t \int_{B_0^1} F(tx) dx. 
\end{align*}
In particular, the derivative of $\phi$ is nonnegative, so $\phi(t) \geq \phi(0)$ when $t \in [0,r]$.  Expressing the integral of $F(tx)$ in polar coordinates, it follows that
\[ \frac{d \phi}{dt}(t) \geq \rho t \int_0^1 u^{d-1} \phi(t u) du \geq \frac{\rho \phi(0) t}{d} \]
and the fundamental theorem of calculus then asserts that
\[ \phi(t) = \phi(0) + \int_0^t \frac{d \phi}{dt}(t) dt \geq \left( 1 + \frac{\rho t^2}{2d} \right) \phi(0). \]
Reversing the polar coordinates gives
\[ \int_{B_0^1} F(t x) dx = \int_0^1 u^{d-1} \phi(tu) du \geq \frac{\phi(0)}{d} + \frac{\phi(0) \rho t^2}{2 d(d+2)}. \]
Lastly, \eqref{conseq} follows from the observation that $d^{-1} \phi(0) = m(B_0^1) F(0)$.
As far as decay of $F$ is concerned, note that the left-hand side is uniformly bounded when $t=1$ for all $x$ in $|x| > R+1$ when $F \in L^1(\R^d)$.  Consequently $F \in L^\infty$.  The exponential decay follows by fixing $t := (\rho^{-1}(d+2))^{-1/2}$ and noting that \eqref{conseq} also implies that the supremum of $F(x)$ on $|x| \leq T$ is bounded above by $\frac{3}{2}$ times the supremum of $F(x)$ on $|x| \leq T - t$ whenever $T - t > R$.
\end{proof}
Note that, just like the mean value property, if $F$ is $C^2$, nonnegative and \eqref{conseq} holds for all $t$ sufficiently small, then the constraint \eqref{subsol} must hold at the point $x$.  However, since \eqref{connect} or \eqref{connectalt} need only be satisfied on unit balls, they allow for more general sorts of weights as well.

Now we come to the proof of corollary \ref{bigcor} itself.
Suppose $\Delta [W(x)]^s \geq \rho [W(x)]^s$ for every $x$ outside the ball of radius $R$ centered at the origin and some $\rho$ (this is precisely the content of the hypothesis \eqref{differential}, though the exact values of $\rho$ in both cases differ by a negligible factor of $s$).  Inside the ball $|x| \leq R+1$ suppose as well that $a \leq |W(x)| \leq A$.  Consider the modified weights
\[ \widetilde{W}(x) := \begin{cases} A e^{-3 \lambda (|x| - R - 1)} & |x| \leq R+1 \\ W(x) & |x| > R+1 \end{cases} \]
\[ \widetilde{W}_+(x) := \begin{cases} A e^{-3 \lambda (|x| - R - 1)} & |x| \leq R+1 \\ \left( \frac{2(d+2) + \eta}{2(d+2) + \rho} \frac{1}{m(B_x^1)} \int_{B_x^1} [W(y)]^s dy \right)^{1/s} & |x| > R+1 \end{cases} \]
where $\eta$ is a small, positive constant.  Clearly we have that $\widetilde{W}(x) \leq \widetilde{W}_+(x)$ (since outside the ball of radius $R+1$, the inequality follows from \eqref{conseq} for any $\rho \geq 0$).  This pair $(\widetilde W, \widetilde{W}_+)$ is admissible with respect to the Lebesgue measure and the Euclidean family of balls on $\R^d$.  When $\frac{1}{2} \leq |x| \leq R+1$, the inequality \eqref{connect} may be verified directly for the exponentially decaying weight, and when $|x| > R+1$, we have
\[ \frac{1}{m(B_x^1)} \int_{B_x^1} [W(y)]^s dy \leq \frac{1}{m(B_x^1)} \int_{B_x^1} [\widetilde W(y)]^s dy \]
which implies \eqref{connectalt} for the ball $B_x^1$ as long as $\eta < \rho$, $\lambda$ is sufficiently near $1$ and $\epsilon$ is sufficiently small.  Corollary \ref{bigcor} now follows by applying \eqref{mainest} from the main theorem to the pair $(\widetilde{W},\widetilde{W}_+)$ and noting that $\widetilde{W} \approx W$ and $\widetilde{W}_+ \approx W_+$ uniformly over all $x \in \R^d$.  
As already remarked, the condition $\int_{X_0} f = 0$ may be replaced by $\int f W_{+} = 0$ at no cost.  Also note that the decay condition on $f$ imposed by the main theorem is vacuous in this case since $W$ is necessarily in $L^\infty$. 



\section{Logarithmic Sobolev inequalities}
In this section, we present a self-contained proof of theorem \ref{logsobthm}.  The proof is elementary and short, reducing for the most part to the following lemma. \label{logsobsec}
\begin{lemma}
For any nonnegative, nondecreasing functions $\psi,u$ on $[0,\infty)$.  Suppose $\{a_j\}_{j \geq 0}$ is a convergent complex sequence with limit $L$.  For any $\theta \in (0,1)$, \label{logsoblemma}
\begin{equation} \psi(|L|) u(|L|) \leq \psi \left( \frac{|a_0|}{1-\theta} \right) u(|L|) + \sup_{k \geq 1} \ \psi \left( \frac{|a_{k}|}{1-\theta} \right)   u \!\left( \frac{|a_{k-1}-L|}{\theta} \right). \label{super1} \end{equation}
\end{lemma}
\begin{proof}
Suppose first that $|a_0 - L| \leq \theta |L|$.  This implies by the triangle inequality that $|L| \leq |a_0| + |a_0 - L| \leq |a_0| + \theta |L|$, so in this case we have
\begin{equation} \psi(|L|) u(|L|) \leq \psi \left( \frac{|a_0|}{1-\theta} \right) u(|L|). \label{triv} \end{equation}
Now suppose instead that $|a_0 - L| > \theta |L|$.  Let $k$ be the smallest index greater than zero for which $|a_k - L| \leq \theta |L|$ (since $a_k \rightarrow L$, such an index must exist unless $L = 0$, in which case \eqref{triv} holds).  For this specific index $k$, it must be the case that $|a_{k-1} - L| > \theta |L|$ and $(1-\theta) |L| \leq |a_{k}|$, 
which together give that
\begin{equation} \psi(|L|) u(|L|) \leq \psi \left( \frac{|a_{k}|}{1-\theta} \right) u \left( \theta^{-1} |a_{k-1}-L| \right). \label{sidenote}
\end{equation}
Clearly this establishes \eqref{super1}.
\end{proof}

To establish \eqref{logsob} consider $y$ to be temporarily fixed and let
\[ a_j := c \frac{\int_{U_j^*(y)} f(z) W(z) d \mu(z)}{\int_{U_j^*(y)} W(z) d \mu(z)} \left[ W(y) \right]^{\frac{1}{p}} ||f||_{p,\psi}^{-1} \]
for some constant $c$ to be chosen momentarily (and $||f||_{p,\psi}$ given by \eqref{singnorm}).
Jensen's inequality coupled with \eqref{subsol0} yield that
\[ |a_j| \lesssim c \left[\mu(U_j^*(y))\right]^{-\frac{1}{p}} \left( \int |f|^p W d \mu \right)^{\frac{1}{p}} ||f||_{p,\psi}^{-1}; \]
however \eqref{sobcond1} implicitly requires $W_+ \gtrsim W$ (since $\psi$ is increasing and strictly positive), so by \eqref{themajortheorem} we have
\begin{align*}
 \int |f|^p W & d \mu  \lesssim \int |f|^p W_+ d \mu \\
 & \lesssim \int_X   \left[ \frac{1}{\mu(U_0(x))}  \int_{U_0(x)} |f(x) - f(y)|^p d \mu(y)  \right] W(x) d \mu(x) \\
& \lesssim \int_{X} \left( \int_{U_0(x)} |f(x) - f(y)|^p K_{p,\psi}(x,y)  d \mu(y) \right) W(x) d \mu(x) = ||f||_{p,\psi}^p
\end{align*}
since $\mu(U_0(x)) \leq K_{p,\psi}(x,y)$.  It follows then that $|a_j| \lesssim c \left[\mu(U_j^*(y))\right]^{-\frac{1}{p}}$ uniformly in $y$ and $f$.  In particular, let $c$ be chosen sufficiently small that $|a_j| \leq \left[ \mu(U_j^*(y))\right]^{-\frac{1}{p}}$.

Now define $F(y) := c f(y) \left[ W(y) \right]^{\frac{1}{p}} ||f||_{p,\psi}^{-1}$.  We wish to apply lemma \ref{logsoblemma} with $u(t) := t^p$ (and $\theta = \frac{1}{2}$, say) to conclude that
\begin{equation}
\begin{split}
\psi( |F(y)|) \left| F(y) \right|^p   \leq  \psi & \left( \left[\mu(U_0^*(y))\right]^{-\frac{1}{p}} \right) \left| F(y) \right|^p \\ & + \sup_{j \geq 0}  \psi \left( \left[\mu(U_{j+1}^*(y))\right]^{-\frac{1}{p}} \right) \left| a_j - F(y)\right|^p. \end{split} \label{super2}
\end{equation}
The application is fairly immediate, thanks to the monotonicity of $\psi$ and the inequality $|a_j| \leq [ \mu(U_j^*(y))]^{-1/p}$, as long as it is known that $a_j \rightarrow F(y)$ as $j \rightarrow \infty$.   If $\vols(y,y) = 0$, then the convergence may be {\it assumed}, since the right-hand side of \eqref{super2} will be infinite when convergence does not hold (thanks to the fact that $\mu(U_j^*(y)) \rightarrow 0$ and $\psi$ tends to infinity at infinity).  In the remaining case, $\vols(y,y) > 0$, Lebesgue dominated convergence together with the fact that $\cap_j U_j^*(y) = \{y\}$ dictates that $a_j \rightarrow c f(y) \left[ W(y) \right]^{\frac{1}{p}} ||f||_{p,\psi}^{-1}$ as $j \rightarrow \infty$ (in particular, the limit is finite since the singleton set $\{y\}$ has $\mu$-positive measure in this case and $f$ belongs to $L^p$ with the weight $W$).  Thus \eqref{super2} must always hold. 

Now Jensen's inequality applied to $|a_j - F(y)|^p$ gives that
\begin{align*} |a_j - F(y)|^p & \leq \frac{c^p W(y)}{||f||_{p,\psi}^p} \frac{\int_{U_j^*(y)} |f(y) - f(z)|^p W(z) d \mu(z)}{\int_{U_j^*(y)} W(z) d \mu(z)}  \\
& \lesssim \frac{c^p}{||f||_{p,\psi}^p} \left( \mu(U_j^*(y)) \right)^{-1} \int_{U_j^*(y)} |f(y) - f(z)|^p W(z) d \mu(z) \\
& \lesssim \frac{c^p}{||f||_{p,\psi}^p} \int_{U_j^*(y)} |f(y) - f(z)|^p \frac{W(z) d \mu(z)}{\vols(z,y)}.
\end{align*}
Now for any fixed $j$, then, the monotonicity of $\psi$ and the containment relations of the $U_j^*$'s dictate that 
\begin{equation}
\begin{split} 
\psi  \left( \left[\mu(U_{j+1}^*(y))\right]^{-\frac{1}{p}} \right) & \left| a_j - F(y)\right|^p \\  \lesssim \frac{c^p}{||f||_{p,\psi}^p}  \int_{U_j^*(y)} & |f(y) - f(z)|^p \psi \left( \left[ \vols(z,y) \right]^{-\frac{1}{p}} \right) \frac{W(z) d \mu(z)}{\vols(z,y)}. 
\end{split} \label{sidenote2}
\end{equation}
Thus the integral over $y$ of this supremum term on the right-hand side of \eqref{super2} is at most a uniform constant times $c^p$.
On the other hand, 
\[ \psi \left( \left[\mu(U_0^*(y))\right]^{-\frac{1}{p}} \right) \left| F(y) \right|^p \lesssim \frac{W_+(y)}{W(y)} c^p |f(y)|^p W(y) ||f||_{p,\psi}^{-p} \]
(the introduction of $W_+/W$ being possible by \eqref{sobcond1}).  Exploiting \eqref{themajortheorem} again, the integral of the right-hand side of \eqref{super2} is at most a uniform constant times $c^p$.  Thus
\[ \int_X \psi( |F(y)|) \left| F(y) \right|^p d \mu(y) \lesssim c^p. \]
Finally, note that
\[ \psi( \left[ W(y) \right]^{-p} |F(y)|) \lesssim \psi(|F(y)|) + \tilde \psi( \left[ W(y) \right]^{-\frac{1}{p}} ), \]
so that we may conclude 
\begin{align*}
 \int_X \psi \left( \frac{c |f(y)|}{||f||_{p,\psi}} \right) & \left|  \frac{c |f(y)|}{||f||_{p,\psi}} \right|^p W(y) d \mu(y)  \\ & \lesssim \int_X \psi( |F(y)|) \left| F(y) \right|^p d \mu(y) \\ & \qquad + \int_X \tilde \psi( \left[ W(y) \right]^{-\frac{1}{p}} ) \left|  \frac{c |f(y)|}{||f||_{p,\psi}} \right|^p W(y) d \mu(y).
 \end{align*}
Recalling \eqref{sobcond1} and \eqref{themajortheorem} a final time, the second term on the right-hand side is again uniformly bounded by a fixed constant times $c^p$.  Choosing $c$ sufficiently small establishes \eqref{logsob}.

Regarding the final comment from section \ref{commentsec}, the needed improvement of lemma \ref{logsoblemma} is as follows.
Suppose $G \subset {\mathbb Z}_{\geq 0}^2 \cap \set{(k,j) \in \Z^2}{ k > j} $ contains all pairs $(k,k-1)$ for $k \geq 1$.  For any $k > 0$, let $G_k := \set{j \in {\mathbb Z}_{\geq 0}}{ (k,j) \in G}$.  Then under the same hypotheses as lemma \ref{logsoblemma},
\begin{equation*} \psi(|L|) u(|L|) \leq \psi \left( \frac{|a_0|}{1-\theta} \right) u(|L|) + \sup_{k \geq 0} \frac{1}{\# G_k} \sum_{j \in G_k} \psi \left( \frac{|a_k|}{1-\theta} \right) u \left( \frac{|u_j - L|}{\theta} \right). 
\end{equation*}
The proof is essentially exactly the same as the proof of \eqref{super1}, since when $k$ is the minimal index for which $|a_k - L| \leq \theta |L|$, then it is actually true that $|a_{j} - L| > \theta |L|$ for all $j < k$ instead of merely for $j = k-1$.  In particular, one may take \eqref{sidenote}, replace $k-1$ with $j$, then average $j$ over all of $G_k$.  This in turn leads one to generalize the inequality corresponding to \eqref{sidenote2} by replacing $j+1$ with $k$ and then averaging over $j \in G_k$.  Now the extra factor of $(\#G_k)^{-1}$ gives more decay and allows one to improve $K_{p,\psi}$.  Specifically one should choose
\[  G := \set{(k,j)}{\psi \left( \left[ \mu(U_{j+1}^*(y)) \right]^{-\frac{1}{p}} \right) \geq  \epsilon \psi \left( \left[ \mu(U_{k}^*(y)) \right]^{-\frac{1}{p}} \right) \mbox{ and } k > j} \]
for some fixed, small $\epsilon$.



\bibliography{mybib}

\def\cprime{$'$}
\providecommand{\bysame}{\leavevmode\hbox to3em{\hrulefill}\thinspace}
\providecommand{\MR}{\relax\ifhmode\unskip\space\fi MR }
\providecommand{\MRhref}[2]{%
  \href{http://www.ams.org/mathscinet-getitem?mr=#1}{#2}
}
\providecommand{\href}[2]{#2}
\begin{thebibliography}{10}

\bibitem{af2003}
Robert~A. Adams and John J.~F. Fournier, \emph{Sobolev spaces}, second ed.,
  Pure and Applied Mathematics (Amsterdam), vol. 140, Elsevier/Academic Press,
  Amsterdam, 2003. \MR{2424078 (2009e:46025)}

\bibitem{ams1963}
N.~Aronszajn, F.~Mulla, and P.~Szeptycki, \emph{On spaces of potentials
  connected with {$L^{p}$} classes}, Ann. Inst. Fourier (Grenoble) \textbf{13}
  (1963), 211--306. \MR{0180846 (31 \#5076)}

\bibitem{bbcg2008}
Dominique Bakry, Franck Barthe, Patrick Cattiaux, and Arnaud Guillin, \emph{A
  simple proof of the {P}oincar\'e inequality for a large class of probability
  measures including the log-concave case}, Electron. Commun. Probab.
  \textbf{13} (2008), 60--66. \MR{2386063 (2009d:60039)}

\bibitem{bcg2008}
Dominique Bakry, Patrick Cattiaux, and Arnaud Guillin, \emph{Rate of
  convergence for ergodic continuous {M}arkov processes: {L}yapunov versus
  {P}oincar\'e}, J. Funct. Anal. \textbf{254} (2008), no.~3, 727--759.

\bibitem{cgww2009}
Patrick Cattiaux, Arnaud Guillin, Feng-Yu Wang, and Liming Wu, \emph{Lyapunov
  conditions for super {P}oincar\'e inequalities}, J. Funct. Anal. \textbf{256}
  (2009), no.~6, 1821--1841.

\bibitem{chafai2004}
Djalil Chafa{\"{\i}}, \emph{Entropies, convexity, and functional inequalities:
  on {$\Phi$}-entropies and {$\Phi$}-{S}obolev inequalities}, J. Math. Kyoto
  Univ. \textbf{44} (2004), no.~2, 325--363. \MR{2081075 (2005e:60170)}

\bibitem{gagliardo1958}
Emilio Gagliardo, \emph{Propriet\`a di alcune classi di funzioni in pi\`u
  variabili}, Ricerche Mat. \textbf{7} (1958), 102--137. \MR{0102740 (21
  \#1526)}

\bibitem{gi2008}
Ivan Gentil and Cyril Imbert, \emph{The {L}\'evy-{F}okker-{P}lanck equation:
  {$\Phi$}-entropies and convergence to equilibrium}, Asymptot. Anal.
  \textbf{59} (2008), no.~3-4, 125--138. \MR{2450356 (2010h:60217)}

\bibitem{gs2010III}
Philip~T. Gressman and Robert~M. Strain, \emph{Sharp anisotropic estimates for
  the {B}oltzmann collision operator and its entropy production},
  ar{X}iv:1007.1276, to appear in Adv. in Math.

\bibitem{gs2010}
\bysame, \emph{Global classical solutions of the {B}oltzmann equation with
  long-range interactions}, Proc. Nat. Acad. Sci. \textbf{107} (2010), no.~13,
  5744--5749.

\bibitem{gs2011}
\bysame, \emph{Global classical solutions of the {B}oltzmann equation without
  angular cut-off}, J. Amer. Math. Soc. \textbf{24} (2011), 709--769.

\bibitem{gross1975}
Leonard Gross, \emph{Logarithmic {S}obolev inequalities}, Amer. J. Math.
  \textbf{97} (1975), no.~4, 1061--1083.

\bibitem{hajlasz2001}
Piotr Haj{\l}asz, \emph{Sobolev inequalities, truncation method, and {J}ohn
  domains}, Papers on analysis, Rep. Univ. Jyv\"askyl\"a Dep. Math. Stat.,
  vol.~83, Univ. Jyv\"askyl\"a, Jyv\"askyl\"a, 2001, pp.~109--126. \MR{1886617
  (2003a:46052)}

\bibitem{lsy2003}
Elliott~H. Lieb, Robert Seiringer, and Jakob Yngvason, \emph{Poincar\'e
  inequalities in punctured domains}, Ann. of Math. (2) \textbf{158} (2003),
  no.~3, 1067--1080. \MR{2031861 (2004m:26025)}

\bibitem{lv2009}
John Lott and C{\'e}dric Villani, \emph{Ricci curvature for metric-measure
  spaces via optimal transport}, Ann. of Math. (2) \textbf{169} (2009), no.~3,
  903--991.

\bibitem{mrs2011}
Cl\'{e}ment Mouhot, Emmanuel Russ, and Yannick Sire, \emph{Fractional
  {P}oincar\'{e} inequalities for general measures}, J. Math. Pures Appl.
  \textbf{95} (2011), no.~1, 72--84.

\bibitem{stein1962}
E.~M. Stein, \emph{The characterization of functions arising as potentials.
  {II}}, Bull. Amer. Math. Soc. \textbf{68} (1962), 577--582. \MR{0142980 (26
  \#547)}

\bibitem{sturm2006I}
Karl-Theodor Sturm, \emph{On the geometry of metric measure spaces. {I}}, Acta
  Math. \textbf{196} (2006), no.~1, 65--131.

\bibitem{sturm2006II}
\bysame, \emph{On the geometry of metric measure spaces. {II}}, Acta Math.
  \textbf{196} (2006), no.~1, 133--177.

\bibitem{taibleson1963}
M.~H. Taibleson, \emph{Lipschitz classes of functions and distributions in
  {$E_{n}$}}, Bull. Amer. Math. Soc. \textbf{69} (1963), 487--493. \MR{0150581
  (27 \#577)}

\bibitem{villani2009}
C{\'e}dric Villani, \emph{Hypocoercivity}, Mem. Amer. Math. Soc. \textbf{202}
  (2009), no.~950, iv+141. \MR{2562709}

\bibitem{wu2000}
Liming Wu, \emph{A new modified logarithmic {S}obolev inequality for {P}oisson
  point processes and several applications}, Probab. Theory Related Fields
  \textbf{118} (2000), no.~3, 427--438. \MR{1800540 (2002f:60109)}

\end{thebibliography}
\end{document}